\documentclass[11pt]{article}
\usepackage{amsmath, amsthm, amssymb,latexsym,mathscinet}
\usepackage{enumerate}
\usepackage{hyperref}
\usepackage{caption,subcaption,xcolor}
\usepackage{tikz,graphicx,cite}
\begin{document}
\baselineskip 16pt

\newtheorem{theorem}{Theorem}[section]
\newtheorem{lemma}[theorem]{Lemma}
\newtheorem{claim}[theorem]{Claim}
\newtheorem{prop}[theorem]{Proposition}
\newtheorem{conjecture}{Conjecture}
\newtheorem{example}{Example}
\newtheorem{fact}[theorem]{Fact}
\newtheorem{remark}{Remark}
\newtheorem{corollary}[theorem]{Corollary}
\newtheorem{construction}[theorem]{Construction}
\newtheorem{proposition}{Proposition}
\newtheorem{question}{Question}

\setlength{\unitlength}{11pt}
\renewcommand{\labelenumi}{(\theenumi)}

\def\blue#1{\textcolor{blue}{#1}}
\def\red#1{\textcolor{red}{#1}}

\title{\bf\vspace{-1.5cm}  On incidence choosability of cubic graphs\footnote{This research supported by Basic Science Research Program through the National Research Foundation of Korea (NRF) funded by the Ministry of Science, ICT and Future Planning (NRF-2018R1C1B6003577).}}
\author{Sungsik Kang and Boram Park\footnote{Corresponding author: borampark@ajou.ac.kr} \\[1.5ex]
\small Department of Mathematics, Ajou University, Suwon 16499, Republic of Korea}

\maketitle
\date

\begin{abstract} \normalsize
An \textit{incidence} of a graph $G$ is a pair $(u,e)$ where $u$ is a vertex of $G$ and $e$ is an edge of $G$ incident with $u$. Two incidences $(u,e)$ and $(v,f)$ of $G$ are \textit{adjacent} whenever (i) $u=v$, or (ii) $e=f$, or (iii) $uv=e$ or $uv=f$. An \textit{incidence $k$-coloring} of $G$ is a mapping from the set of incidences of $G$ to a set of $k$ colors such that every two adjacent incidences receive distinct colors. The notion of incidence coloring has been introduced by Brualdi and Quinn Massey (1993) from a relation to strong edge coloring, and since then, attracted by many authors.

On a list version of incidence coloring, it was shown by Benmedjdoub et. al. (2017) that every Hamiltonian cubic graph is incidence 6-choosable. In this paper, we show that every cubic (loopless) multigraph is incidence 6-choosable.
As a direct consequence, it implies that the list strong chromatic index of a $(2,3)$-bipartite graph is at most 6, where a (2,3)-bipartite graph is a bipartite graph such that one partite set has maximum degree at most 2 and the other partite set has maximum degree at most 3.
\end{abstract}

\section{Introduction}\label{sec:Intro}
In this paper, we always consider a loopless graph.
For a graph allowed to have a multiple edge, we call it a \textit{multigraph}.
In addition, we call a simple graph, a graph without multiple edge, just a \textit{graph}.

An \textit{incidence} of a multigraph $G$ is a pair $(u,e)$ of a vertex $u$ and an edge $e$ such that $u$ is an endpoint of $e$. In this case, we say $(u,e)$ is an incidence on the edge $e$.
We denote by $I(G)$ the set of all incidences of a multigraph $G$.
Two incidences $(u,e)$ and $(v,f)$ are \textit{adjacent} if one of the following holds: (i) $u=v$; (ii) $e=f$; (iii) $uv=e$ or $uv=f$. See Figure~\ref{fig:def}.
\begin{figure}
  \centering
  \includegraphics[width=14cm]{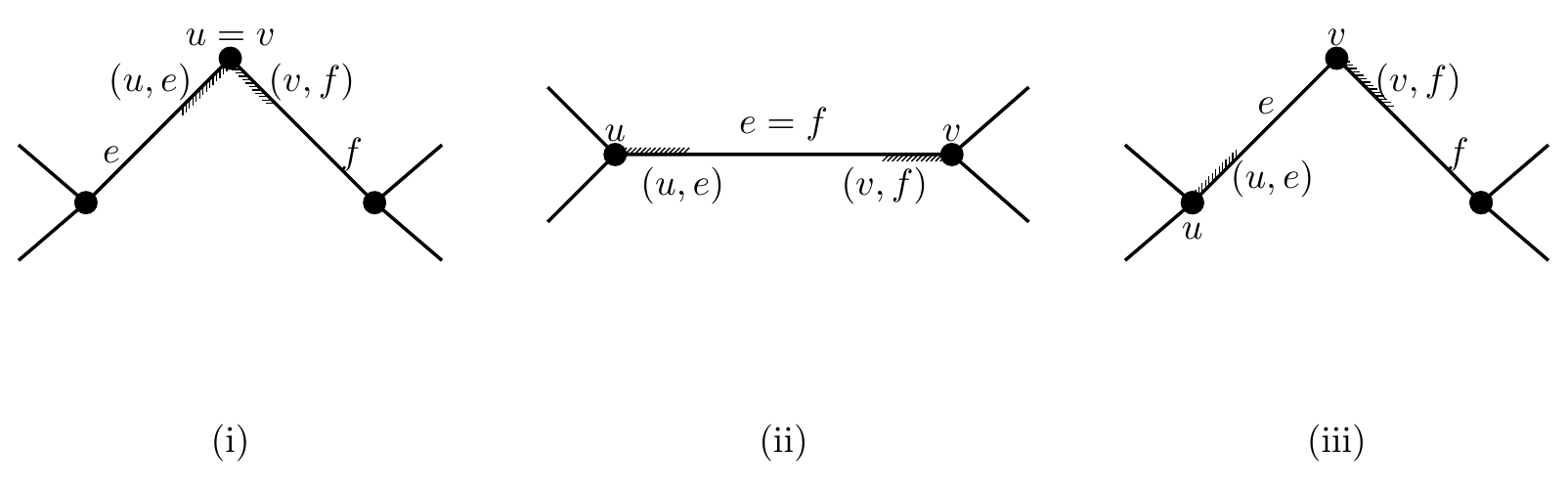}\\
  \caption{Illustrations of adjacent incidences. The shaded parts show two adjacent incidences.}\label{fig:def}
\end{figure}
An \textit{incidence coloring} of a multigraph $G$ is a function $\varphi$ from $I(G)$ to the set of colors so that $\varphi(u,e)\neq \varphi(v,f)$ for any two adjacent incidences $(u,e)$ and $(v,f)$ of $G$.
An \textit{incidence $k$-coloring} of a multigraph $G$ means an incidence coloring of $G$ using at most $k$ colors, and
the \textit{incidence chromatic number} $\chi_i(G)$ of  $G$ is the smallest integer $k$ such that $G$ admits an incidence $k$-coloring.

The notion of incidence coloring was introduced by Brualdi and Quinn Massey~\cite{BQ1993}, and has attracted by interesting questions in a relationship with another coloring problems.
For example, an incidence coloring is interpreted as a special case of directed star arboricity problem, introduced by Algor and Alon in \cite{AA1989} (see \cite{Gu1997} for more details).
In addition, the incidence chromatic number of a graph $G$ gives a lower bound for $\chi(G^2)$, where $G^2$ is the square of $G$ (see \cite{Wu2009}).

It is worthy to see a relation to strong edge coloring, since the notion of incidence coloring was motivated by strong edge coloring, in its introduction.
A \textit{strong edge coloring} of a graph $G$ is an edge coloring such that each color class is an induced matching, and the \textit{strong chromatic index} $\chi'_s(G)$ is the smallest integer $k$ such that $G$ allows a strong edge coloring using $k$ colors. There are intensive research on this topic started from one famous conjecture of Erd\H{o}s and Ne\v{s}et\v{r}il (see \cite{FSGT1990}).
In~\cite{BQ1993}, it was known that an incidence coloring of a multigraph $G$ corresponds to a strong edge coloring of
the \textit{subdivision} $S(G)$ of $G$, where $S(G)$ is the graph obtained from $G$ by subdividing every edge of $G$ (see Figure~\ref{fig:relation}). More explicitly, $\chi_i(G)=\chi'_s(S(G))$ for any multigraph $G$.

Note that the subdivision of a multigraph with maximum degree $\Delta$ is a $(2,\Delta)$-bipartite graph, where $(a,b)$-bipartite graph $G$ means a bipartite graph with a bipartition $(A,B)$ such that $\max\{\deg_G(v)\mid v\in A\} \le a$ and $\max\{\deg_G(v)\mid v\in B\} \le b$.
In this context, Brualdi and Quinn Massey suggested the following Conjecture~\ref{conj:bipartite}, which is a refinement of a conjecture by Faudree, Gy$\acute{\text{a}}$rf$\acute{\text{a}}$s, Schelp and Tuza in \cite{FSGT1990}:
\begin{conjecture}[\cite{BQ1993}]\label{conj:bipartite}
For every $(a,b)$-bipartite graph $G$, $\chi'_s(G)\le ab$.
\end{conjecture}
The conjecture is still open, and see \cite{N2008, NN2014, BF2015, BLV2016, HYZ2017} for some recent partial results. We remark that in \cite{N2008,HYZ2017}, it was shown that the conjecture holds if $a=2$ or $3$.
\begin{figure}
  \centering
  \includegraphics[width=10cm]{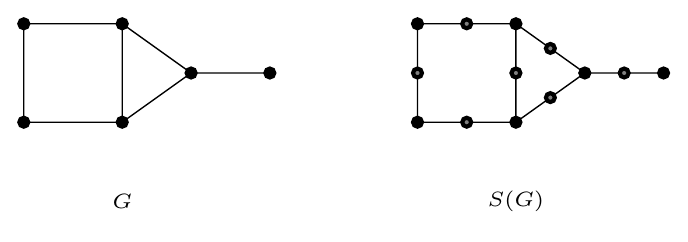}\\
  \caption{Graphs $G$ and $S(G)$}\label{fig:relation}
\end{figure}

It is easy to observe that for any connected graph $G$ with at least three vertices, $\Delta(G)+1\le \chi_i(G)\le 3\Delta(G)-2$.
In \cite{BQ1993}, the authors showed that
$\chi_i(G)\le 2\Delta(G)$, and conjectured that any graph $G$ satisfies that $\chi_i(G)\le \Delta(G)+2$. This conjecture was called the incidence coloring conjecture, but few years later, counterexamples were found in \cite{Gu1997}. Even though it turned out be false, it leads to several interesting questions about graphs satisfying the equality.
On the class of cubic graphs, in \cite{SLC2002}, it was shown that some several classes of cubic graphs are  incidence 5-colorable. In 2005, Maydanskiy \cite{M2005} showed that any cubic graph is incidence 5-colorable:
\begin{theorem}[\cite{M2005}]\label{thm:cubic}
For every cubic graph $G$, $\chi_i(G)\leq 5$.
\end{theorem}
Hence, every cubic graph has incidence chromatic number 4 or 5, and it is known to be NP-hard to determine if the incidence chromatic number of a cubic graph is 4 (see \cite{LP2008, JMM2017}).
See \cite{SS2008, SW2013, MGS2012, GLS2016, GLS2017, J2018} for some interesting families of graphs whose incidence chromatic numbers are known.

In this paper, we study a list version of incidence coloring.
For a multigraph $G$,
an \textit{incidence list assignment} $L$ of $\red{G}$ is a function defined on $I(G)$
such that $L$ associates each $(v,e)\in I(G)$ with a nonempty set $L(v,e)$.
If $|L(v,e)|=k$ for each $(v,e)\in I(G)$, then we say it is an \textit{incidence $k$-list assignment}.
When an incidence list assignment $L$ of a multigraph $G$ is given,
we say $G$ is \textit{incidence $L$-choosable}
if there is an incidence coloring $\varphi$ such that $\varphi(v,e)\in L(v,e)$ for each incidence $(v,e)$ of $G$, and in this case such $\varphi$ is called an \textit{incidence $L$-coloring}.
A multigraph $G$ is said to be incidence \textit{$k$-choosable} if $G$ is incidence $L$-choosable for any incidence $k$-list assignment. The \textit{incidence
choice number} $ch_i(G)$ of a multigraph $G$ is the smallest integer $k$ such that $G$ is incidence $k$-choosable.

Whereas the incidence chromatic number for a cubic graph is known to be at most 5 as in Theorem~\ref{thm:cubic}, the incidence choice number of a cubic graph is firstly studied in  \cite{BBS2017} most recently as follows.

\begin{theorem}[\cite{BBS2017}]\label{thm:main:Hamiltonian}
For every Hamiltonian cubic graph $G$, $ch_i(G)\leq 6$.
\end{theorem}

In this paper, we generalize Theorem~\ref{thm:main:Hamiltonian}, as follows:

\begin{theorem}[Main result]\label{thm:main}
For a subcubic loopless multigraph $G$, $ch_i(G)\le 6$.
\end{theorem}

Note that the bound 6 is tight by considering the graph on two vertices with three multiple edges.

By a relation to strong edge coloring, we immediately obtain the following. The \textit{list strong chromatic index} $ch'_s(G)$ of a graph $G$ is the smallest integer $k$ such that for every list assignment $L$ on $E(G)$ satisfying that $|L(e)|\ge k$ for any $e\in E(G)$, there is a strong edge coloring $\varphi$ of $G$ so that $\varphi(e)\in L(e)$ for every edge $e$ of $G$.

\begin{theorem}[Equivalent to Theorem~\ref{thm:main}]\label{cor:main}
For every $(2,3)$-bipartite graph $G$, $ch'_s(G)\leq 6$.
\end{theorem}

\begin{proof} Let $G$ be a $(2,3)$-bipartite graph with bipartition $(A,B)$, say
$\max\{\deg_G(v)\mid v\in A\}\le 2$.
For every vertex of degree 1 in $A$,
we add a new (pendent) vertex $v'$ so that $vv'$ is an edge. If we let $G'$ be the resulting graph, then $G'$ is a (2,3)-bipartite graph so that every vertex of one partite set has degree exactly two.
Since $G$ is a subgraph of $G'$, it is sufficient to show that $ch'_s(G')\leq 6$.
Let $H$ be the multigraph such that  $S(H)=G'$, which is
obtained from $G'$ by suppressing every vertex in $A$.
Then $H$ is a loopless subcubic multigraph, and so $ch_i(H)=ch'_s(G')$.
Thus Theorem~\ref{thm:main} implies Theorem~\ref{cor:main}.
By a similar way, we can also
prove that Theorem~\ref{cor:main} implies Theorem~\ref{thm:main} here.
\end{proof}

Here are some  terminologies used in this paper.
A \textit{semicubic} graph is a graph such that each vertex has degree 1 or 3.
A $(1,2)$-\textit{decomposition} of a graph $G$ is a bipartition $(E_1,E_2)$ of the edge set $E(G)$ such that $E_1$ and $E_2$ induces a 1-factor and a 2-regular graph, respectively.
Note that any graph having a (1,2)-decomposition is semicubic.
We say an incidence $L$-coloring $\varphi$ of a subgraph of a multigraph $G$ is \textit{well-extended} if there is an incidence $L$-coloring $\varphi'$ of $G$ as an extension of $\varphi$.

The paper is organized as follows. Section~\ref{sec:main} outlines the proof of the main theorem first. Section~\ref{sec:prel} collects some known observations which will be used in our proof.
Finally Section~\ref{sec:proofs} completes the proof.

\section{Outline of the proof of Theorem~\ref{thm:main}}\label{sec:main}

We will prove the following two theorems, Theorems~\ref{thm:PM} and~\ref{the:submain} in Section~\ref{sec:proofs}.
\begin{theorem}\label{thm:PM}
For a 2-connected cubic graph $G$, $ch_i(G)\le 6$.
\end{theorem}

For a graph $G$  and an incidence list assignment $L$ of $G$,
an edge $e=uv$ of $G$ is said to be \textit{freely $L$-choosable} if for any two distinct colors $c_1\in L(u,e)$ and $c_2\in L(v,e)$,
there is an incidence $L$-coloring $\varphi$ such that $\varphi(u,e)=c_1$ and $\varphi(v,e)=c_2$.

\begin{theorem}\label{the:submain} Let $G$ be a semicubic graph with at least one pendent vertex, $L$ be an incidence $6$-list assignment of $G$. Suppose that the graph obtained from $G$ by deleting all pendent vertices of $G$ is $K_1$ or 2-connected.
Then every pendent edge  of $G$ is freely $L$-choosable.
  \end{theorem}

We also mention one fact in graph theory.

\begin{fact}\label{fact}
For a cubic graph $G$, $G$ has a cut-vertex if and only if it has a cut edge.
\end{fact}

Here, we give a proof of the main result, using Theorems~\ref{thm:PM} and ~\ref{the:submain}.

\begin{proof}[Proof of Theorem~\ref{thm:main}]
Let $G$ be a minimum counterexample with respect to the number of vertices.
Then there is an incidence $6$-list assignment of $G$ such that
$G$ is not incidence $L$-choosable. Note that $G$ is connected. First, we will show the following.

\begin{claim}\label{claim:min:counterex}
The graph $G$ is cubic and has no multiple edge.
\end{claim}

\begin{proof}[Proof of Claim~\ref{claim:min:counterex}]
Suppose that $G$ has a vertex $v$ of degree at most $2$.
By minimality,  $G-v$, which is the graph obtained from $G$ by deleting the vertex $v$, has an incidence $L$-coloring $\varphi$.
Let $e_1=vu_1$ and $e_2=vu_2$ be the edges incident to $v$. (Note that it may hold that $u_1=u_2$.)
Then the number of available colors at each of $L(v,e_1)$  and $L(v,e_2)$ is at least four
and the number of available colors at each of $L(u_1,e_1)$  and $L(u_2,e_2)$ is at least two.
By choosing colors from $L(u_1,e_1)$  and $L(u_2,e_2)$ first,
 it is easy to see that $\varphi$ is well-extended to an incidence $L$-coloring of $G$, which is a contradiction. Hence, $G$ is a cubic.

Next, we will show that $G$ is a simple graph. Suppose that there is a pair of vertices $u$ and $v$ having multiple edges.
If there are three multiple edges between $u$ and $v$, then $G$ has only two vertices $u$ and $v$ and three multiple edges, and so $|I(G)|=6$ and $G$ is incidence $6$-choosable.
Suppose that there are exactly two multiple edges $e_1$ and $e_2$ between $u$ and $v$.
Let $u_1$ and $v_1$ be the another neighbor of $u$ and $v$, respectively.
Let $G'=G-\{u,v\}$, which is the graph obtained from $G$ by deleting the vertices $u$ and $v$. By minimality of $G$, $G'$ has an incidence $L$-coloring $\varphi$.
In the following, for an incidence $(x,f)$ of $G$ not in $G'$,
we let $L'(x,f)$ be the set of available colors at that incidence under $\varphi$.
Note that  
for an incidence $(x,f)$ of $G$ not in $G'$, the following hold.
For every incidence $(x,f)$ on an edge joining $u$ and $v$,  $|L'(x,f)|\ge 6$.
For the other incidences,
if $u_1\neq v_1$, then
\[ |L'(u_1,uu_1)|, |L'(v_1,vv_1)|\ge 2, \quad |L'(u,uu_1)|,|L'(v,vv_1)|\ge 4,\]
and if  $u_1= v_1$, then
\[ |L'(u_1,uu_1)|, |L'(v_1,vv_1)|\ge 4, \quad |L'(u,uu_1)|,|L'(v,vv_1)|\ge 5.\]
See Figure~\ref{fig:counterex}.
\begin{figure}[h!]
\begin{center}
\includegraphics[width=10cm]{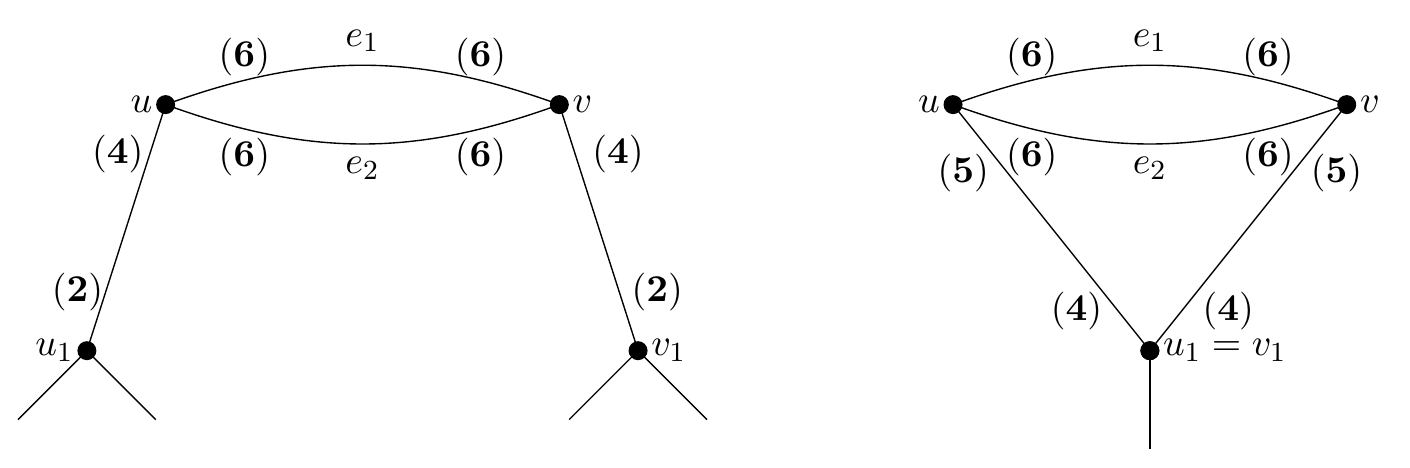}
\caption{An illustration for Claim~\ref{claim:min:counterex}. The bold number in the parenthesis on each incidence $(x,f)$ denotes a lower bound of $|L'(x,f)|$.}\label{fig:counterex}\end{center}
\end{figure}

\smallskip

\noindent Case (i) Suppose that there is a color $c\in L'(u,uu_1)\cap L'(v,vv_1)$. Then after giving the color $c$ to both $(u,uu_1)$ and $(v,vv_1)$, note that there are a sufficient number of available colors at each of $L'(u_1,uu_1)$  and $L'(v_1,vv_1)$.
More precisely, if $u_1\neq v_1$, then there is at least one available color in each  of $L'(u_1,uu_1)-\{c\}$ and  $L'(v_1,vv_1)-\{c\}$, and the colors for them may be same.
If $u_1=v_1$, then there are at least three available colors  in each  of $L'(u_1,uu_1)-\{c\}$ and  $L'(v_1,vv_1)-\{c\}$ and the colors for them should be distinct.
Then the number of the remaining incidences is  four, and
the number of available colors at each of them is also four.
Thus, we can extend $\varphi$ to an incidence $L$-coloring of $G$.

\smallskip

\noindent Case (ii) Suppose that  $L'(u,uu_1)\cap L'(v,vv_1)=\emptyset$.
Then $|L'(u,uu_1)\cup L'(v,vv_1)|\ge 8$, and so there is a coloring $c\in (L'(u,uu_1)\cup L'(v,vv_1))-L'(u,e_1)$. If $c\in L'(u,uu_1)$, then we give the color $c$ to $(u,uu_1)$.
If $c\in L'(v,vv_1)$, then we give the color $c$ to $(v,vv_1)$.
Then, there are a sufficient number of available colors at each of $L'(u_1,uu_1)$  and $L'(v_1,vv_1)$. If $u_1\neq v_1$, then there is at least one available color.
If $u_1=v_1$, then there are at least two available colors.
Then the number of available colors at $(u,e_1)$ is four, and those of the other remaining  incidences are at least three, and so we can extend $\varphi$ to an incidence $L$-coloring of $G$.

\smallskip

Hence, in any case, we can find an incidence $L$-coloring of $G$, which is a contradiction.
\end{proof}

By Claim~\ref{claim:min:counterex}, $G$ is simple and cubic.
If $G$ is 2-connected, then it holds by Theorem~\ref{thm:PM}.
Suppose that $G$ is not 2-connected. Then by Fact~\ref{fact}, $G$ has cut edges, and let $K$ be the set of all cut edges of $G$ (and so $K\neq \emptyset$). Let $G_1$, $\ldots$, $G_m$ be the components of $G-K$, which is the graph obtained from $G$ by deleting the edges in $K$. Note that each $G_i$ is $K_1$ or 2-connected. For each $i\in \{1,\ldots,m\}$, we let $G^*_i$ be the graph obtained from $G_i$ by adding all edges in $K$ incident to a vertex of $G_i$. See Figure~\ref{fig:outline}.
\begin{figure}[h!]
  \centering
\includegraphics[width=16cm]{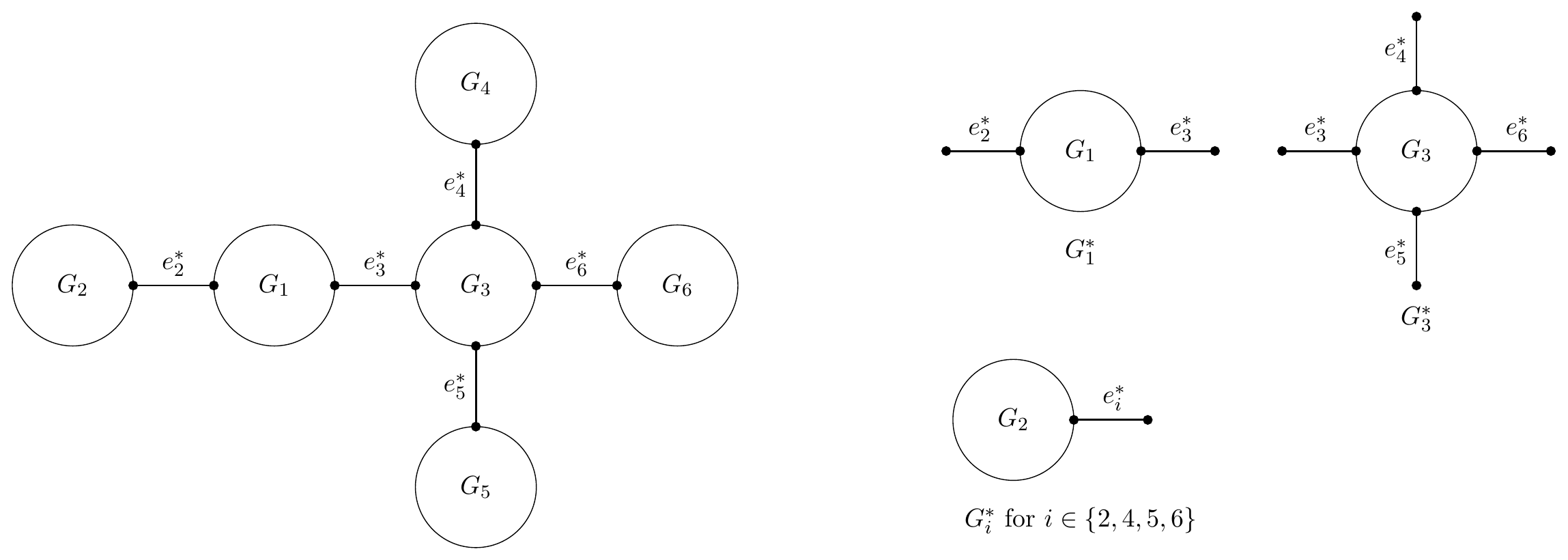}\\
  \caption{An illustration for graphs $G_i$ and $G_i^*$ in the proof of Theorem~\ref{thm:main}}\label{fig:outline}
\end{figure}

Let $\mathcal{G}$ be the graph whose vertices are $G_1,\ldots,G_m$, and $G_i$ and $G_j$ are adjacent if and only if there is a cut edge between a vertex of $G_i$ and a vertex of $G_j$.
Then it is easy to see that  $\mathcal{G}$ is a tree and the number of edges in $G$ (actually in $K$) between $G_i$ and $G_j$ is at most one.
Thus we can relabel the vertices of $\mathcal{G}$ so that for each $i\in\{2,\ldots,m\}$, $G_i$ is adjacent to exactly one $G_j$ for some $j\in \{1,\ldots, i-1\}$, and so we let $e^*_i$ be the edge of $G^*_i$ between $G_i$ and such $G_j$. See Figure~\ref{fig:outline} for an  illustration.

Note that each $G^*_i$ is semicubic at least one pendent edge, and
the graph   obtained from $G^*_i$ by deleting all pendent vertices    is $G_i$, which is $K_1$ or 2-connected. By Theorem~\ref{the:submain},  $G^*_1$ has an incidence $L$-coloring.
Suppose that $G^*_1 \cup \cdots \cup G^*_i$ has  an incidence $L$-coloring $\varphi$ for some $i\in \{1,\ldots,m\}$.
Then among the incidences of $G^*_{i+1}$,  only two incidences on the edge  $e^*_{i+1}$ are precolored under $\varphi$. By applying Theorem~\ref{the:submain} to $G^*_{i+1}$ and  $e^*_{i+1}$, it follows that $\varphi$ is well-extended to an incidence $L$-coloring  of $G^*_1 \cup \cdots \cup G^*_i\cup G^*_{i+1}$. Hence, by repeating the arguments, we obtain an incidence $L$-coloring  of $G$.
\end{proof}

\section{Preliminaries}\label{sec:prel}
We note that from now on, any graph always means a `simple' graph.
In this section, we collect some observations.
The first lemma easily follows from the fact that each incidence of a subcubic graph $G$ has at most 7 forbidden colors from its adjacent incidences.

\begin{lemma}\label{lem:sub}
Let $G$ be a subcubic graph, and $L$ be an incidence $6$-list assignment of $G$.
Let $\varphi$ be an incidence $L$-coloring of a subgraph $H$ of $G$.
If $e,e'\in E(G)\setminus E(H)$ and $e$ is incident to $e'$, then there is an extension $\varphi'$ of $\varphi$ so that $\varphi'$ is an incidence $L$-coloring of $H+e$. (See Figure~\ref{fig:lem}.)
\end{lemma}

\begin{figure}[h!]
\begin{center}
\includegraphics[width=14cm]{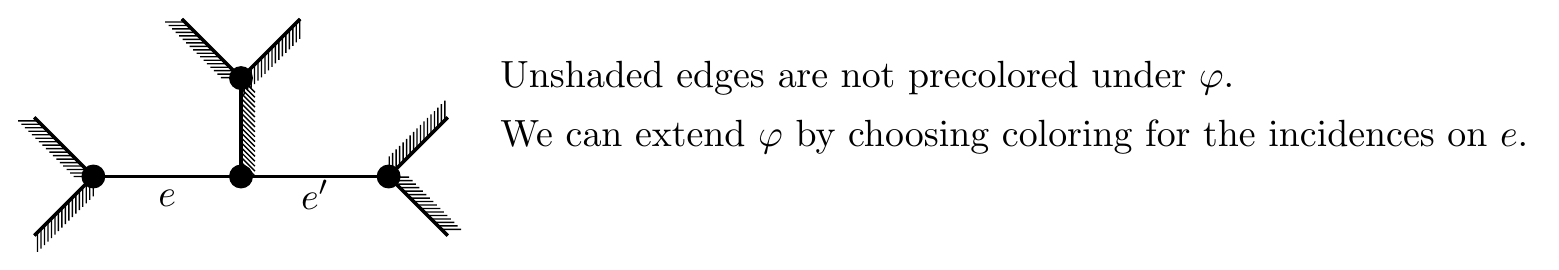}
\caption{An illustration for Lemma~\ref{lem:sub}}\label{fig:lem}\end{center}
\end{figure}

The following lemma was presented in \cite{BBS2017}, and we write its proof here for the sake of completeness.

\begin{lemma}[Claim 3, \cite{BBS2017}]\label{lem:list}
Let $G$ be a subcubic graph, $L$ be an incidence $6$-list assignment of $G$.
Let $v_0$, $v_1$, $v_2$ be three vertices so that $v_0v_1, v_1v_2\in E(G)$ and $v_0v_2\not\in E(G)$. Suppose that for each $i\in\{0,1\}$, there is an edge $e_{v_i}=v_iu_i$ incident to $v_i$ so that
 $e_{v_0}$ and $e_{v_1}$ do not share an endpoint and $u_1\neq v_2$.
(See Figure~\ref{fig:list}.)
Then there exist $\alpha \in L(u_0,e_{v_0})$, $\beta \in L(v_0,e_{v_0})$, $\gamma \in L(u_1,e_{v_1})$, $\delta \in L(v_1,e_{v_1})$, and $\eta \in L(v_2,v_1v_2)$ with $\alpha \neq \beta$, $\delta \neq \gamma$, and $\delta \neq \eta$ satisfying the following:
\begin{eqnarray}\label{condition:twoedges}
&&  |L( v_0,v_0v_1 ) \cap \{\alpha,\delta\}| \leq 1\quad\text{and}\quad
|L( v_1,v_0v_1) \cap \{\beta,\gamma,\eta\}| \leq 1.
\end{eqnarray}
\end{lemma}

\begin{figure}[h!]
\centering
  \includegraphics[scale=1]{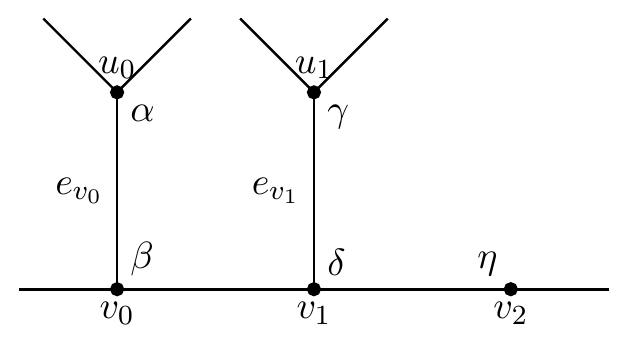}
\caption{An illustration for Lemma~\ref{lem:list}}\label{fig:list}
\end{figure}

\begin{proof}
We first deal with the incidence $(v_1,v_0v_1)$.
For simplicity, let $B = L(v_0,e_{v_0})$, $ C= L(u_1,e_{v_1})$, and $E = L(v_2,v_1v_2)$. If $B\cap C\cap E \neq \emptyset$, then we set $\beta = \gamma = \eta = \mu$ for some $\mu \in  B\cap C\cap E$, and so \eqref{condition:twoedges} holds for
$(v_1,v_0v_1)$.
Suppose that $B\cap C\cap E=\emptyset$.
If $B$, $C$, and $E$ are pairwise disjoint, then each of $B\cup C$, $C\cup E$, $B\cup E$ has 12 elements, and so for some two $X,Y\in\{B,C,E\}$ so that
$|X-L(v_1,v_0v_1)|\ge 1$ and $|Y-L(v_1,v_0v_1)|\ge 1$, which implies that can choose $\beta$, $\gamma$, and $\eta$ in such a way that $|L( v_1,v_0,v_1) \cap \{\beta,\gamma,\eta\}| \leq 1$.
Suppose that $B\cap C \neq \emptyset$ (the cases $B\cap E \neq \emptyset$ and $C\cap E \neq \emptyset$ are similar). We first set $\beta = \gamma = \mu$ for some $\mu \in B\cap C$.
If $\mu \not\in L( v_1,v_0v_1)$, then for any $\eta\in E$, $|L( v_1,v_0v_1) \cap \{\beta,\gamma,\eta\}| \leq 1.$
If $\mu \in L(v_1,v_0v_1 )$, then there exists $\eta \in E\setminus L( v_1,v_0v_1)$ (note that $E\setminus L( v_1,v_0v_1)=\emptyset$ implies that $\mu\in B\cap C\cap E$, a contradiction.) and so that $|L( v_1,v_0v_1) \cap \{\beta,\gamma,\eta\}| \leq 1$.

We now deal with the incidence $(v_0,v_0v_1)$.
For simplicity, let $A' = L(u_0,e_{v_0})\setminus \{\beta\}$ and $D' = L(v_1,e_{v_1})\setminus\{ \gamma,\eta\}$. Note that $|A'| \geq 5$ and $|D'| \geq 4$.
If $A'\cap D' \neq \emptyset$, then we set $\alpha = \delta = \mu$ for some $\mu \in A'\cap D'$, so that $ |L(v_0,v_0v_1 ) \cap \{\alpha,\delta\}| \leq 1.$
Suppose that $A'\cap D' = \emptyset$.
Then $|A'\cup D'| \geq 9$. Then $A'\cup D'-L(v_0,v_0v_1)\neq \emptyset$, and so take $\mu \in A'\cup D'-L(v_0,v_0v_1)$.
Note that either $\mu\in A'$ or $\mu\in D'$.
If $\mu\in A'$, then let  $\alpha=\mu$, and so $ |L(v_0,v_0v_1 ) \cap \{\alpha,\delta\}| \leq 1$ for any  $\delta$.
If $\mu\in D'$, then we set $\delta = \mu$, so that $ |L(v_0,v_0v_1 ) \cap \{\alpha,\delta\}| \leq 1$ for any  $\alpha$.
This completes the proof.
\end{proof}

We mention a result in \cite{M2005} and add its proof here for the sake of completeness.

\begin{lemma}[Proposition 3, \cite{M2005}]\label{lem:12decomp}
Let $G$ be a semicubic graph with exactly one pendent vertex $v$ such that $G-v$ is 2-connected.
Then $G$ has a (1,2)-decomposition.
\end{lemma}
\begin{proof} We will show that $G$ satisfies Tutte's condition.
Let $q(H)$ be the number of odd components of a graph $H$.
Let $S \subseteq V(G)$, $C$ be an odd component of $G-S$.
Since the sum of degrees of the vertices of $G$ is  even, $G$ has an odd number of edges between $S$ and $C$.
Since $G$ has only one bridge, $G-S$ has at least 3 such edges except at most one connected component. The number of edges between $S$ and $G-S$ is at least $3(q(G-S)-1)+1=3q(G-S)-2$. Then $3q(G-S)-2 \leq 3|S|$ and so $q(G-S) \leq |S|$. Therefore $G$ satisfies the Tutte's condition.
\end{proof}

We finish the section by mentioning a result of incidence choosability of a cycle.
\begin{lemma}[Theorem 5, \cite{BBS2017}]\label{lem:cycle}
For every cycle $C$ of length $n$, $ch_i(C)\le 4$ and $ch_i(C)=3$ if and only if $n$ is a multiple of $3$.
\end{lemma}

\section{Proofs of Theorems~\ref{thm:PM} and ~\ref{the:submain}}\label{sec:proofs}

First, we define several notions for the proof.
A graph in which of each block is either a cycle or a $K_2$ is said to be \textit{cycle-tree} if no two pendent edges are adjacent. Note that an edge of a cycle-tree not on a cycle is a cut edge.  We call a cycle $C$ of a  cycle-tree  a \textit{pendent} if the cycle is a pendent block.

\begin{figure}[b!]
  \centering
  \includegraphics[width=17cm]{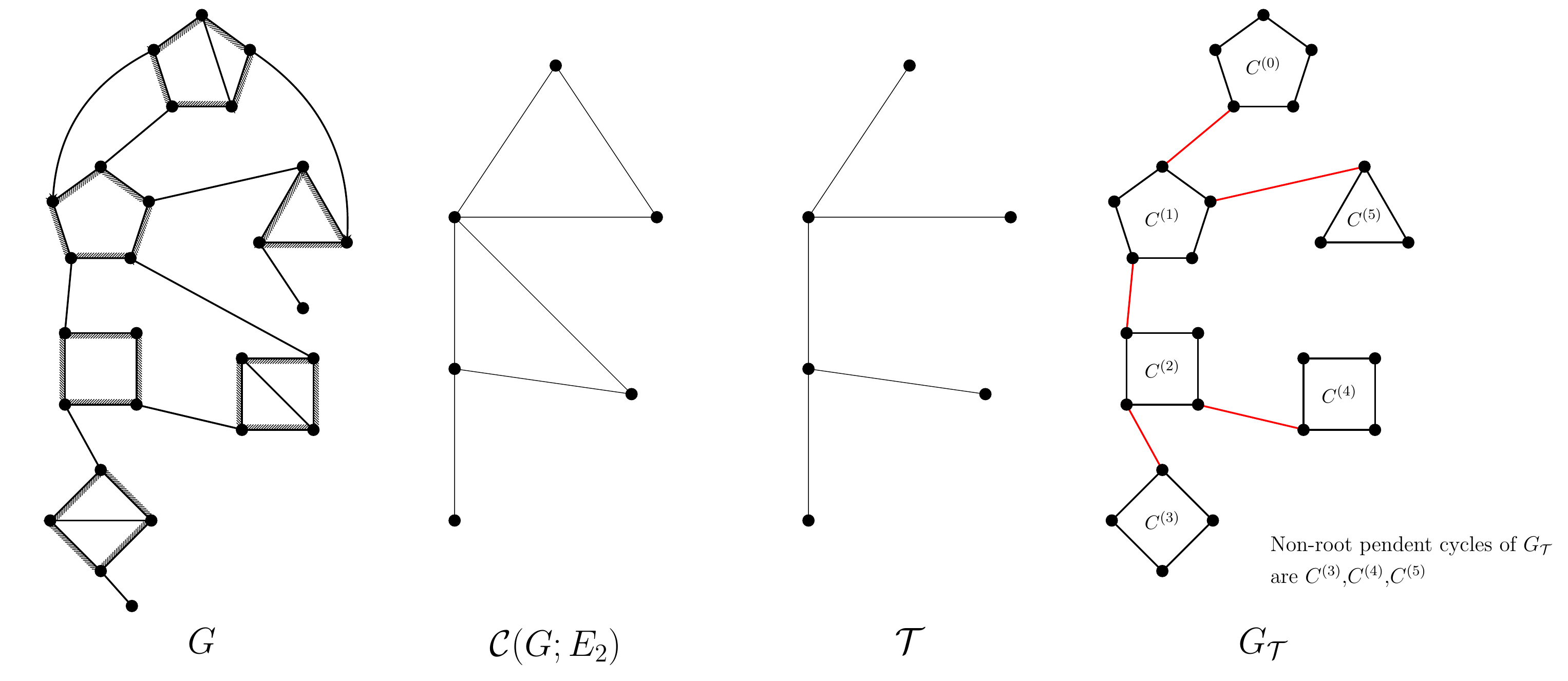}\\
\caption{A semicubic graph $G$ with a (1,2)-decomposition $(E_1,E_2)$ ($E_2$ are the shaded edges in the leftmost figure), the graph $\mathcal{C}(G;E_2)$, a spanning tree $\mathcal{T}$ of $\mathcal{C}(G;E_2)$,  a cycle-tree $G_\mathcal{T}$ with the structure $\mathcal{T}$}\label{fig:cycle-tree2}
\end{figure}

Let $G$ be a connected graph having a (1,2)-decomposition $(E_1,E_2)$.
For simplicity, let  $\mathcal{C}(G;E_2)$ be the graph whose vertices are corresponding to the cycles of $G[E_2]$ and two vertices corresponding to two cycles $C$ and $C'$ of $G[E_2]$ are adjacent if and only if there is an edge in $G$ between $C$ and $C'$.
Then from a spanning tree $\mathcal{T}$ of $\mathcal{C}(G;E_2)$, we naturally obtain
a minimal cycle-tree $G_\mathcal{T}$ as a subgraph of $G$ containing the graph $G[E_2]$ (This may not be unique by the choice of  cut edges of the cycle-tree.), which comes from the definition of $\mathcal{C}(G;E_2)$.
In this case,  this cycle-tree $G_\mathcal{T}$ is a cycle-tree with the structure $\mathcal{T}$.
See Figure~\ref{fig:cycle-tree2}.

Recall that Depth-First-Search (DFS for short) is an algorithm to find a spanning tree of a graph, (called a DFS-tree in the following) and it gives a natural vertex ordering of the graph and we call it a DFS ordering. For instance, an ordering $C^{(0)}$, $C^{(1)}$, \ldots, $C^{(5)}$ of the cycles in Figure~\ref{fig:cycle-tree2} is a DFS-ordering. Let $\mathcal{T}$ be a DFS-tree with an ordering $C^{(0)}$, $C^{(1)}$, \ldots, $C^{(m)}$ and  $G_\mathcal{T}$ be a cycle-tree with the structure $\mathcal{T}$.
For a non-root cycle $C^{(i)}$ of $G_\mathcal{T}$,
the \textit{parent vertex} $p(C^{(i)})$ of $C^{(i)}$ is defined to be a vertex on $C^{(j)}$ for some $j<i$, which is an endpoint of the cut edge joining
a vertex $C^{(i)}$ and a vertex $C^{(j)}$.

First we state a useful lemma, which will be proved later.
An \textit{almost cubic graph} is a graph such that every vertex has degree three, except at most one vertex.

\begin{lemma}\label{lem:one-cycle}
Let $G$ be an almost cubic connected graph with a (1,2)-decomposition $(E_1,E_2)$ and $L$ be an incidence $6$-list assignment of $G$.
Let $G_\mathcal{T}$ be a cycle-tree with the structure $\mathcal{T}$, where $\mathcal{T}$ is a DFS-tree of $\mathcal{C}(G;E_2)$. Then $G$ is  $L$-choosable.
Moreover, an edge $e^*$ incident to the root cycle $C^{(0)}$ is freely $L$-choosable if one of the following cases holds (see Figure~\ref{fig:Hamiltonian22}):
\begin{itemize}
\item[\rm{(i)}] $e^*=x_0x_1$ is an edge on $C^{(0)}:x_0x_1\ldots, x_{n-1}$ $(n\ge 7)$
    such that $x_2x_4, x_3x_5, x_1x_6 \in E(G)$;
\item[\rm{(ii)}] $e^*$ is a pendent edge;
\item[{\rm(iii)}] $e^*$ is not a cut edge of  $G_{\mathcal{T}}$,
$e^*$ is incident to  a non-root pendent 6-cycle $x_1x_2\ldots,x_6$ of $G_{\mathcal{T}}$ where $x_2x_4,x_3x_5\in E(G)$ and
$x_1$ is an endpoint of $e^*$.
\end{itemize}
\end{lemma}

\begin{figure}[b!]\centering
  \includegraphics[width=15.6cm]{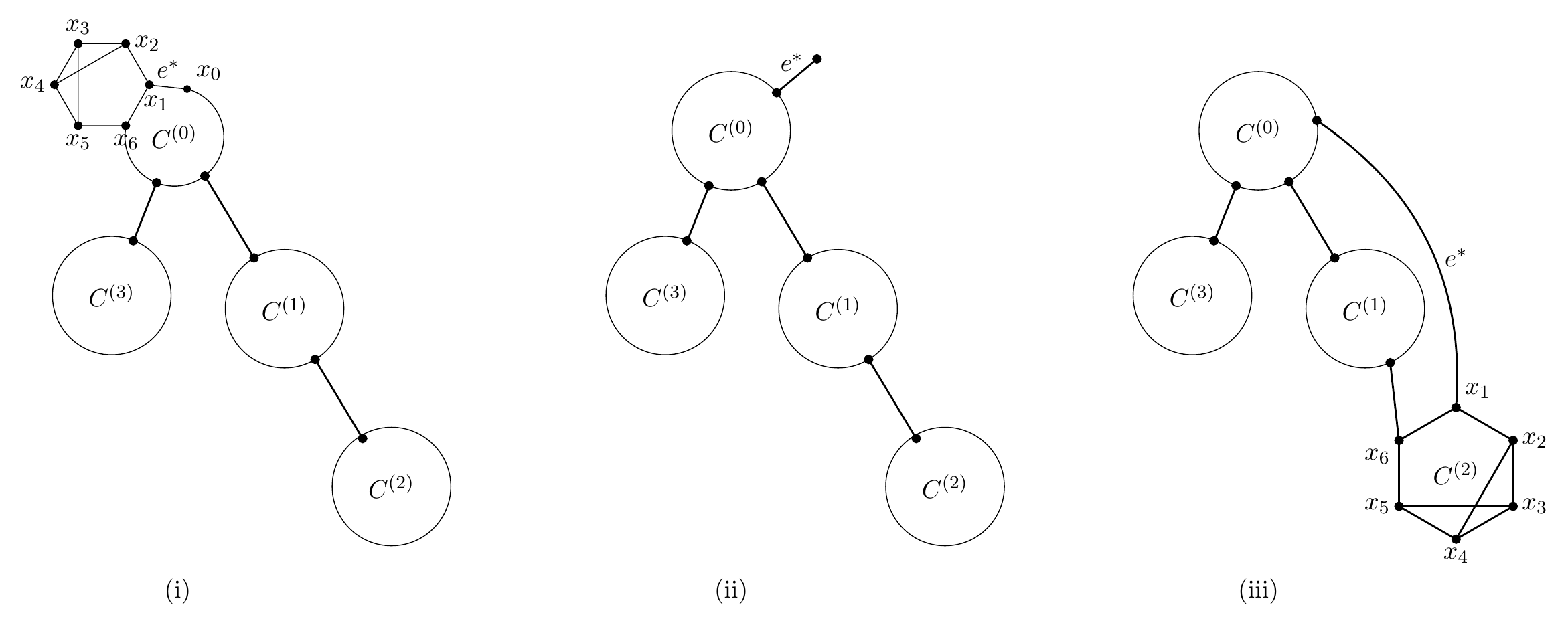}\\
  \caption{An illustration for Lemma~\ref{lem:one-cycle}}\label{fig:Hamiltonian22}
\end{figure}

We also mention a famous theorem called Petersen's theorem.

\begin{theorem}[Petersen's theorem \cite{p1891}]\label{thm:Petersen}
A 2-connected cubic graph contains a perfect matching.
\end{theorem}

Now we are ready to prove the theorems.

\begin{proof}[Proof of Theorem~\ref{thm:PM}]
Let $G$ be a 2-connected cubic graph.
By Theorem~\ref{thm:Petersen}, $G$ has a perfect matching, and so it has a (1,2)-decomposition.
By Lemma~\ref{lem:one-cycle}, Theorem~\ref{thm:PM} follows.
\end{proof}

\begin{proof}[Proof of Theorem~\ref{the:submain}] Take a pendent edge $e^*$.
Let $G_0$ be the graph obtained from $G$ by deleting all pendent vertices of $G$.
First, suppose that $G_0=K_1$.
Then $G$ is $K_2$ or $K_{1,3}$, and it is easy to see that the theorem is satisfied.
Now suppose that $G_0$ is 2-connected.
We will define an almost cubic graph $G'$ with at most one pendent vertex as follows.
For two pendent vertices $x$ and $y$ of $G$, we attach the gadget graph as in Figure~\ref{fig:gadget} so that $x$ and $y$ are identified with the vertices of degree two in the gadget. We repeat this process until we have at most one pendent vertex in $G$. In addition, we can find such $G'$ so that $e^*$ is the remaining pendent edge of $G$ (if it has.).

\begin{figure}
  \centering
 \includegraphics[width=14cm]{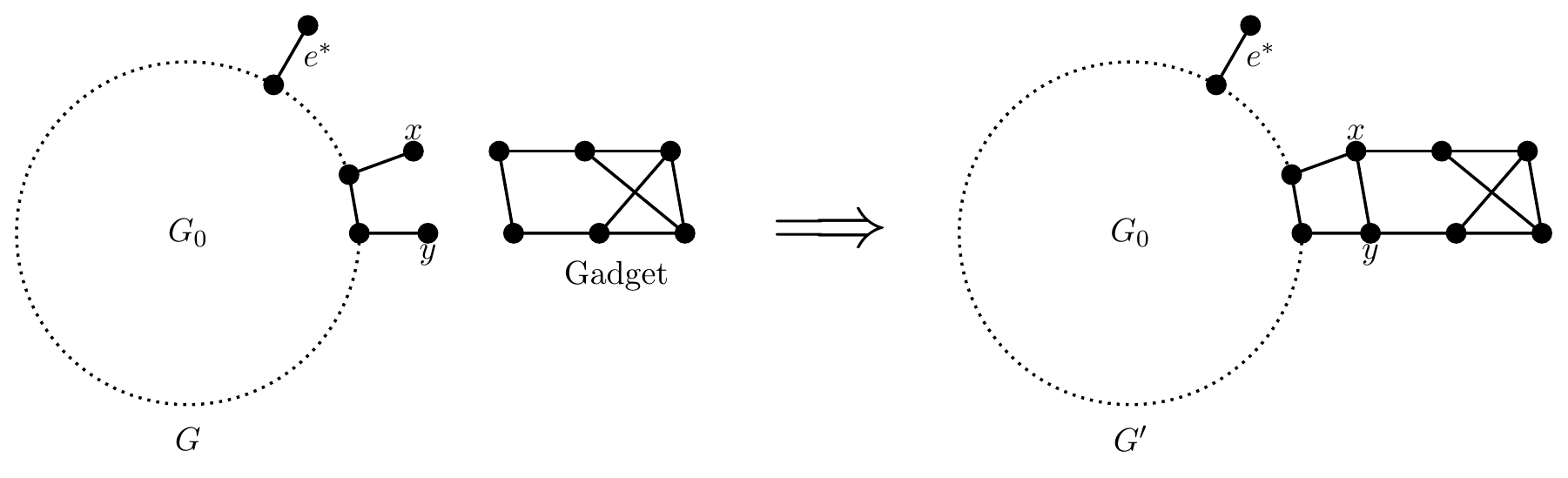}\\
  \caption{An illustration for the proof of Theorem~\ref{the:submain}}\label{fig:gadget}
\end{figure}

Let $G'$ be a resulting graph and $L'$ be an incidence $6$-list assignment of $G'$ which is an extension of $L$.
Note that  $G'$ is a semicubic graph with at most one pendent vertex,
and $G'$ minus the pendent vertex is 2-connected. By Lemma~\ref{lem:12decomp}
$G'$ has a (1,2)-decomposition ($E_1',E_2'$).
Let $\mathcal{C}'=\mathcal{C}(G',E_2')$ for simplicity.
We divide into two cases according to the existence of a pendent vertex in $G'$.

\medskip

\noindent (Case 1) Suppose that $G'$ has a pendent vertex. Then the edge $e^*$ is a pendent edge of $G'$.
We find a cycle-tree $G'_{\mathcal{T}'}$ with the structure $\mathcal{T}'$ of $G'$ so that $\mathcal{T}'$ is a DFS-tree of $\mathcal{C}'$ whose root cycle is the cycle incident to $e^*$. By (ii) of Lemma~\ref{lem:one-cycle},  $e^*$ is freely $L'$-choosable.

\medskip

\noindent (Case 2) Suppose that $G'$ has no pendent vertex. Then there is a  pendent  edge $f^*$ of $G$ such that $f^*$ is paired with $e^*$ in the construction of $G'$. Note that $e^*\in E_1'$ if and only if $f^*\in E_1'$.  Let $H^*$ be the gadget graph attached  to $e^*$ and $f^*$ in the construction of $G'$.
First, suppose that $e^*\in E_1'$. Then $f^*\in E_1'$ and so $e^*$ is incident to two cycles in $G'[E_2']$, and let $C^{(0)}$  be the cycle incident to $e^*$, so that $C^{(0)}$ is the  cycle whose vertices are in $V(G')-V(H^*)$. Take a DFS-tree $\mathcal{T}'$ of $\mathcal{C}'$ whose root cycle is $C^{(0)}$. In addition, we can let the 6-cycle of $G'[E_2']$ contained in $V(H^*)$ a pendent cycle in $\mathcal{T}'$, by considering the cycle incident to $f^*$ first in a DFS-tree $\mathcal{T}'$.
Moreover, in this case, $e^*$ is not an cut edge of $G_{\mathcal{T}}$.
By (iii) of Lemma~\ref{lem:one-cycle},  $e^*$ is freely $L'$-choosable.
Thus, $e^*$ is freely $L$-choosable in $G$.

Now suppose that $e^*\in E_2'$. We may assume that there is a cycle $C^{(0)}$ of $G'[E_2']$ which includes $V(H^*)$ and the edge $e^*$. Note that its length is at least 7.
We take a DFS-tree $\mathcal{T}$ of $\mathcal{C}'$ whose DFS-ordering is starting with $C^{(0)}$.
Then by (i) of Lemma~\ref{lem:one-cycle},  $e^*$ is freely $L'$-choosable.
\end{proof}

We finish the section by giving the proof of Lemma~\ref{lem:one-cycle}.
The idea of the proof  comes from  a cycle-tree obtained by DFS algorithm.

\begin{proof}[Proof of Lemma~\ref{lem:one-cycle}]
We may assume that $G$ has at least 7 vertices.
To see why, suppose that $G$ has at most 6 vertices.
If $G$ is one of $K_4$, the graphs $G_1$ and $G_2$ in Figure~\ref{fig:small}, then it is Hamiltonian cubic, and so $ch_i(G)\le 6$ by Theorem~\ref{thm:main:Hamiltonian}.
In addition, when $G=K_4$, $G_1$ or $G_2$ in Figure~\ref{fig:small}, none of its edges is satisfying the cases in the 'moreover` part of the lemma.
Suppose that $G$ is none of $K_4$, $G_1$ and $G_2$.
Then, since $G$ is almost cubic with a (1,2)-decomposition $(E_1,E_2)$,
$G$ is the graph $G_3$ in Figure~\ref{fig:small}, and this is the case (ii).
\begin{figure}[h!]
  \centering
  \includegraphics[width=8cm]{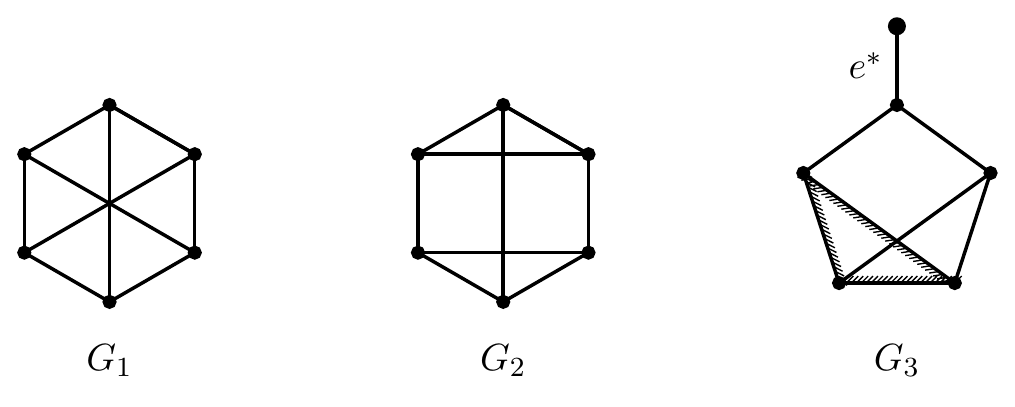}\\
  \caption{The shaded edges indicate the cycle $C^*$ in the proof of Lemma~\ref{lem:one-cycle}.}\label{fig:small}
\end{figure}
Then we can always find a forest $F$ containing $e^*$ so that $G-F$ is a cycle $C^*$ of length 3. Then we choose colors for the incidences on the edges of $F$ first, and during this choice, we can always choose colors for the incidences on $e^*$ first. Each of the incidences on the remaining edges (edges on $C^*$) has at least three available colors and $C^*$ is a cycle of length a multiple of 3. Thus by Lemma~\ref{lem:cycle}, the lemma follows.

In the following, we assume that $|V(G)|\ge 7$. Since $E_1$ is a 1-factor of $G$, we denote  by  $e_v$ the edge of $E_1$ incident to $v$, for  each vertex $v$ of $G$.
\begin{claim}\label{claim:taking:triple}
If $C$ is either a unique cycle of $G_{\mathcal{T}}$ or a non-root pendent cycle of $G_{\mathcal{T}}$, then
there is a triple $(v_0,v_1,v_2)$ of   vertices consecutive  along $C$ satisfying  {\rm(A1)}$\sim$ {\rm(A3)}:
\begin{itemize}
\item[{\rm(A1)}]  $v_0v_2\not\in E(G)$;
\item[{\rm(A2)}]  any two of $e_{v_0}$, $e_{v_1}$, $e^*$ do not share an endpoint;
\item[{\rm(A3)}]  $v_2$ is not an endpoint of $e^*$.
\end{itemize}
In this case, we let
\[I_C =\{ (u,e) \mid e\in \{ e_{v_0},e_{v_1}\}\} \cup \{ (v_2,v_1v_2)\}.\]
Then, moreover, for every two non-root pendent cycles $C$ and $C'$ of $G_{\mathcal{T}}$,
$(v,e)\in I_C$ and $(v',e')\in I_{C'}$ are not adjacent.
\end{claim}

\begin{proof}
First, suppose that $C$ is a unique cycle of $G_{\mathcal{T}}$. Then $C$ has length at least 7.
For the case (i), let $(v_0,v_1,v_2)=(x_4,x_5,x_6)$.
We suppose not the case (i).
Then $G$ has at most one pendent edge, say $e^*$.
Then $C$ minus the endpoints of $e^*$ has at least five vertices  $y_0$, $y_1$, \ldots, $y_4$ consecutive along $C$ so that none of them is not an endpoint of $e^*$.
If $y_0y_2\not\in E(G)$ or $y_1y_3\not\in E(G)$, then either $(y_0,y_1,y_2)$ or $(y_1,y_2,y_3)$ is a desired triple of vertices.
If $y_0y_2\in E(G)$ and  $y_1y_3\in E(G)$, then $(y_4,y_3,y_2)$ is a desired triple.

We suppose that $C$ is a non-root pendent cycle of $G_{\mathcal{T}}$. Note that $m\ge 1$.
Let $q$ be the vertex of $C$ which is incident to a cut edge of $G_\mathcal{T}$.
Suppose the case (iii). That is,  $C:x_1x_2\ldots,x_6$ is a 6-cycle such that  $x_2x_4,x_3x_5\in E(G)$ and $x_1$ is an endpoint of $e^*$.  We take $v_1$ a neighbor of $q$ on $C$, so that $e_{v_1}\neq e^*$. And then let $v_0$ be a neighbor of $v_1$ on $C$ such that $v_0\neq q$. Then $(v_0,v_1,q)$ is a desired triple.
If it is not the case (iii), then we take any triple $(v_0,v_1,q)$ of three vertices consecutive along $C$ (this is possible since $|C|\ge 3$).

The `moreover' part follows from the fact that $\mathcal{T}$ is a DFS-tree, since for a non-root pendent cycle $C$ with a triple $(v_0,v_1,v_2)$ obtained above, each edge $e \in\{ e_{v_0}, e_{v_1}\}$ is either a chord of $C$ or incident to a cycle $C'$ which is not a non-root pendent cycle.
\end{proof}

\begin{figure}[h!]
  \centering
  \includegraphics[width=12cm]{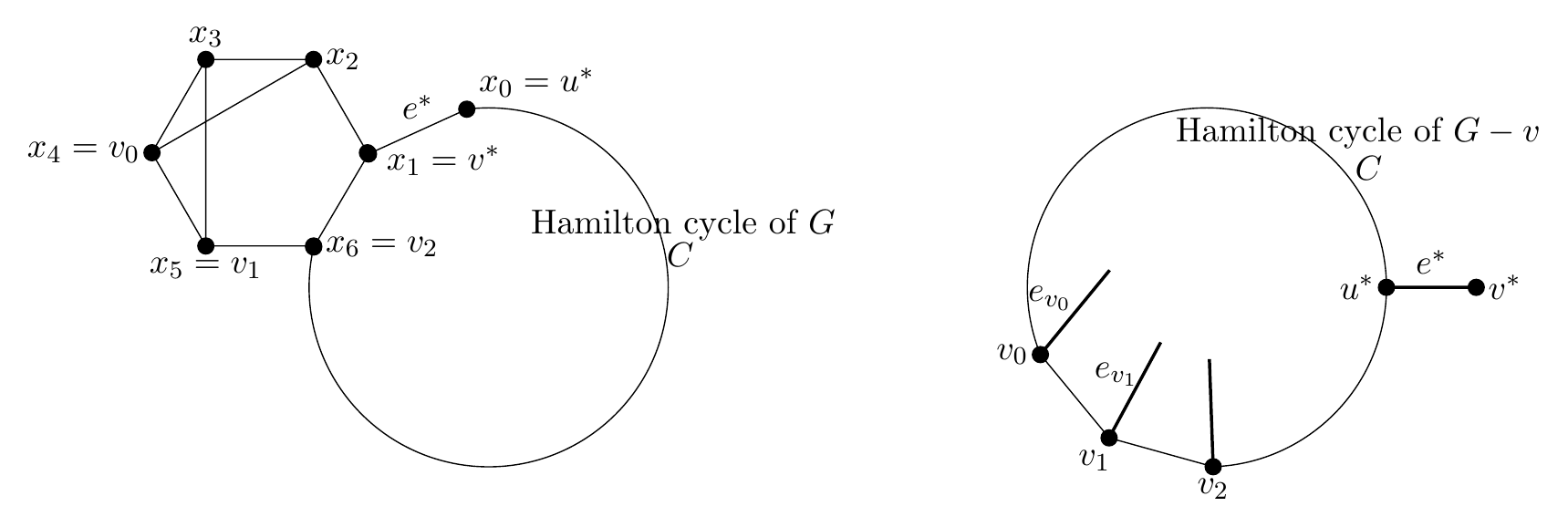}\\
  \caption{An illustration for some cases for the proof of Claim~\ref{claim:taking:triple}}\label{fig:Hamiltonian}
\end{figure}

Now we choose colors for the incidences by the following steps. We state the steps and explain brief reasons.
We let $C^{(0)}$, $\ldots$, $C^{(m)}$ be a DFS-ordering where $m\ge 0$.
Let $M$ be the set of the cut edges of $G_{\mathcal{T}}$, and let $e^*=u^*v^*$.

\begin{enumerate}[Step 1.]
\item We choose colors for the incidences on the edge $e^*$ (if such $e^*$ in the cases exists).
\item If $C$ is either a unique cycle or a non-root pendent cycle $C$ of $G_\mathcal{T}$, then letting $(v_0,v_1,v_2)$ be a triple found in Claim~\ref{claim:taking:triple}, we choose colors $\alpha$, $\beta$, $\gamma$, $\delta$, $\eta$ for the five incidences in $I_C$ (defined in Claim~\ref{claim:taking:triple}), so that they satisfy \eqref{condition:twoedges} in Lemma~\ref{lem:list}.
We note that the Step 2 is valid by Claim~\ref{claim:taking:triple}.
\item For case (i),  we choose colors for the incidences on the edges in $E_1 - (M\cup \{x_1x_6\})$.
For the other cases, we choose colors for the remaining incidences on the edges in $E_1-M$.
This step is valid by Lemma~\ref{lem:sub}, since there is some uncolored incident edge in each coloring step.
\item We take a cycle $C^{(i)}$ in $G_{\mathcal{T}}$ with the smallest index $i$ such that an incidence on some edge of $C^{(i)}$ is not colored yet.
    In the case where $C^{(i)}$ is a unique cycle or a non-root pendent cycle of $G_{\mathcal{T}}$,
    let $(v_0,v_1,v_2)$ be the triple found in Claim~\ref{claim:taking:triple}.
    In the following, except the incidences $(v_1,v_0v_1)$ and $(v_0,v_0v_1)$, existence of remained possible colors is guaranteed by Lemma~\ref{lem:sub}.
\begin{enumerate}
\item Suppose that $i=0$ and $m=0$. See the left figure of Figure~\ref{fig:Hamiltonian}.
    We choose colors for the incidences on the edges on the $(v_2,u^*)$-section of $C$ not containing $v^*$ one by one, starting from the edge $u^*x_{n-1}$. Then we choose colors
    for the incidences on the edge $x_1x_6$ and
    then for the incidence $(v_1,v_1v_2)$, and then for the incidences on the edges $v^*x_2$, $x_2x_3$, $x_3x_4$, one by one.
    We choose colors the incidences $(v_0,v_0v_1)$ and $(v_1,v_0v_1)$ on the edge $v_0v_1$.
Here, we note a brief reason why we can choose those colors.
Even though the second lastly chosen  incidence  $(v_0,v_0v_1)$  has 6 precolored adjacent incidences,  they are using at most five colors by \eqref{condition:twoedges} in Lemma~\ref{lem:list}, and so we still have at least one available color for the incidence $(v_0,v_0v_1)$.
Similarly,  the lastly chosen incidence $(v_1,v_0v_1)$ has 7 precolored adjacent incidences  using at most five colors by \eqref{condition:twoedges} in Lemma~\ref{lem:list}, and so we still have at least one available color for $(v_1,v_0v_1)$.

\item Suppose that $i=0$ and $m\ge 1$. Let $p=p(C^{(1)})$, and note that $p$ is a vertex of $C^{(0)}$.
If it is not the case (i), then we choose colors for the incidences on the edges of $C^{(0)}$ one by one along $C^{(0)}$, starting from the parent vertex $p$.
Suppose that it is the case (i).
Then we consider the path $C^{(0)}-e^*$, and consider the $(u^*,p)$-section $P_1$ and the $(v^*,p)$-section $P_2$ of $C^{(0)}-e^*$.
We choose colors for the incidences on the edges of $P_1$ one by one along $P_1$ starting from the vertex $u^*$, and choose colors for the incidences on the edges of $P_2$ one by one along $P_2$ starting from the vertex $v^*$.

\item Suppose that $C^{(i)}$ is a non-root, non-pendent cycle in $G_{\mathcal{T}}$.
For the parent vertex $p=p(C^{(i)})$ of $C^{(i)}$, first we choose colors for the incidences on $e_p$.
Note that, from the definition of a DFS-tree, there is a cycle $C^{(j)}$ in $G_{\mathcal{T}}$
 so that $j>i$, and let $p'=p(C^{(j)})$.
Then we choose colors for the incidences on the edges of $C^{(i)}$ one by one along $C^{(i)}$, starting from the vertex $p'$.

\item Suppose that $C^{(i)}$ is a non-root pendent cycle in  $G_{\mathcal{T}}$.
 Then we choose colors for the incidences on $e_p$ where $p=p(C^{(i)})$ first, and then choose
 a  color for the incidence $(v_1, v_1v_2)$, and then we determine the remaining incidences on the edges of $C^{(i)}$ one by one along $C^{(i)}$, starting from the vertex $v_2$ so that the incidences on the edge $v_0v_1$ are lastly chosen. Especially, the colors for $(v_1,v_0v_1)$ and $(v_0,v_0v_1)$ can be chosen by \eqref{condition:twoedges} in Lemma~\ref{lem:list} (similar to (a)).
 \end{enumerate}
\item Repeat Step  4 until we choose colors for all incidences of the cycles of $G_{\mathcal{T}}$.
\end{enumerate}
\end{proof}

\bibliographystyle{plain}
\bibliography{Incidence_cubic.bib}

 \end{document}